\documentclass[11pt]{article}
\usepackage[a4paper,margin=30mm]{geometry}
\usepackage{amsmath,amssymb,amsthm,mathtools}
\usepackage{microtype}
\usepackage{booktabs}
\usepackage{tikz}
\usepackage{enumitem}
\usepackage[hidelinks,hypertexnames=false]{hyperref}
\usepackage{url}
\usepackage[T1]{fontenc}
\usepackage[utf8]{inputenc}

\newtheorem{theorem}{Theorem}[section]
\newtheorem{proposition}[theorem]{Proposition}
\newtheorem{lemma}[theorem]{Lemma}
\newtheorem{corollary}[theorem]{Corollary}
\newtheorem{example}[theorem]{Example}
\theoremstyle{remark}
\newtheorem{remark}[theorem]{Remark}

\newcommand{\lcmop}{\operatorname{lcm}}

\newcommand{\Pcal}{\mathcal P}
\newcommand{\1}{\mathbf 1}

\title{Extremal Least Common Multiples in Rows of Pascal's Triangle}
\author{Felix Huber\\
\small Independent researcher, Lucerne, Switzerland\\
\small \texttt{felix.huber.math@proton.me}\\
\small ORCID: \href{https://orcid.org/0009-0005-1568-1579}{0009-0005-1568-1579}}
\date{}

\begin{document}
\maketitle

\begin{abstract}
For $r\ge0$, let
\[
\mathcal P_r=\left\{\binom r0,\binom r1,\ldots,\binom r{\lfloor r/2\rfloor}\right\}
\]
be the set of distinct entries in row $r$ of Pascal's triangle. We study the least possible least common multiple of $n$ entries chosen from one row, with the row itself also free:
\[
a(n)=\min_{\substack{r\ge0,\ S\subseteq\mathcal P_r\\ |S|=n}}\operatorname{lcm}(S).
\]
Whereas the least common multiple of an entire row is given by a classical identity of Farhi, allowing both the row and the selected coefficients to vary creates a different optimization problem. We first recast the fixed-row problem exactly as a weighted prime-power exclusion problem, which explains how omitting a few coefficients can remove expensive prime-power contributions. Our main asymptotic result is
\[
\log a(n)=2n+O\!\left(n\exp\!\left[-c\frac{(\log n)^{3/5}}{(\log\log n)^{1/5}}\right]\right)
\]
for some absolute $c>0$, so $a(n)^{1/n}\to e^2$. A two-band refinement further shows that every optimal row satisfies
\[
r_n=2n+O\!\left(ne^{-c\Phi(n)}\right),
\]
and that the minimum prefix defect of an optimal support is $o(n)$.
\end{abstract}



\section{Introduction and main results}
The least common multiple of a complete row of Pascal's triangle is strikingly simple. If
\[
\Lambda_r=\operatorname{lcm}_{0\le k\le r}\binom rk,
\]
then Farhi \cite{Farhi2009} proved
\begin{equation}
\Lambda_r=\frac{\operatorname{lcm}(1,2,\ldots,r+1)}{r+1}.
\tag{1.1}\label{eq:farhi}
\end{equation}
The complete-row problem goes back at least to Problem E2686 of Montgomery and Breusch \cite{MontgomeryBreusch1979}; Hong \cite{Hong2009} related \eqref{eq:farhi} to an identity of Nair, and Guo \cite{Guo2009} proved a $q$-analogue. A separate classical line, beginning with Kummer and Lucas and including Fine \cite{Fine1947} and Rowland \cite{Rowland2011}, describes the prime and prime-power divisibility pattern inside a row.

We ask what happens when neither the row nor the selected entries are fixed. For
\begin{equation}
\mathcal P_r=\left\{\binom r0,\binom r1,\ldots,\binom r{\lfloor r/2\rfloor}\right\},
\tag{1.2}\label{eq:Pr}
\end{equation}
the set of distinct entries of row $r$, define
\begin{equation}
\boxed{a(n)=\min_{\substack{r\ge0,\ S\subseteq\mathcal P_r\\ |S|=n}}\operatorname{lcm}(S).}
\tag{1.3}\label{eq:defa}
\end{equation}
Equivalently, $a(n)$ is the smallest positive integer having at least $n$ distinct divisors that occur together in one Pascal row. The optimization in \eqref{eq:defa} differs essentially from the classical full-row problem. Farhi's identity determines the LCM once an entire row is prescribed; here both the row and the subset of coefficients are free, and deleting a small number of coefficients may remove large prime or prime-power contributions. Thus the problem involves a competition between moving to a later row, where more coefficients are available, and exploiting additional divisibility among a carefully chosen subset. This is an extremal LCM problem in the spirit of work on subsets of ordinary integer intervals \cite{CillerueloEtAl2014}, but with the rigid arithmetic structure of a Pascal row.

This additional freedom changes the finite structure: optimal supports need not be initial segments of a row. The main results show, nevertheless, that the freedom is asymptotically limited: the minimum LCM has the same leading exponential scale as the complete first admissible row, every optimal row is asymptotically forced back to that scale, and the number of omitted positions before the last selected coefficient is sublinear. We are not aware of a previous study of this simultaneous row-and-subset optimization problem or of an equivalent formulation.

The first terms are
\[
1,\ 2,\ 10,\ 56,\ 252,\ 2310,\ 12012,\ 24024,\ 680680,\ 1007760,\ldots.
\]
The first nontrivial case already shows why the row cannot be fixed in advance. Row $4$ is the first row with three distinct entries,
\[
\mathcal P_4=\{1,4,6\},\qquad \operatorname{lcm}(\mathcal P_4)=12,
\]
but row $5$ gives
\[
\mathcal P_5=\{1,5,10\},\qquad a(3)=10.
\]
Thus moving to a later row can lower the LCM by improving divisibility relations. This example also illustrates the basic tension of the problem: using the first admissible row keeps the coefficients small, whereas a later row may create divisibility relations that lower the LCM. The optimization must balance these two effects.

\begin{figure}[htbp]
\centering
\begin{tikzpicture}[font=\small]
  \node[anchor=east] at (-0.35,1.15) {row $4$};
  \foreach \x/\v in {0/1,1.25/4,2.5/6}{
    \draw[rounded corners=2pt] (\x,0.78) rectangle ++(0.9,0.7);
    \node at (\x+0.45,1.13) {$\v$};
  }
  \node[anchor=west] at (3.9,1.13) {$\operatorname{lcm}=12$};

  \node[anchor=east] at (-0.35,0) {row $5$};
  \foreach \x/\v in {0/1,1.25/5,2.5/10}{
    \draw[rounded corners=2pt] (\x,-0.35) rectangle ++(0.9,0.7);
    \node at (\x+0.45,0) {$\v$};
  }
  \node[anchor=west] at (3.9,0) {$\operatorname{lcm}=10$};
  \draw[->,thick] (6.05,1.02) -- (6.05,0.18);
  \node[align=center,anchor=west,font=\scriptsize] at (6.25,0.60)
    {later row,\\smaller LCM};
\end{tikzpicture}
\caption{The first nontrivial instance of the optimization.  Row $4$ is the first row with three distinct entries, but row $5$ wins because $5\mid10$.  The problem therefore balances coefficient size against divisibility rather than minimizing either one separately.}
\label{fig:firstcompetition}
\end{figure}

Two further phenomena drive the paper. First, although an optimal support may omit strategically expensive coefficients, the complete-row upper bound turns out to have the correct exponential scale. Second, the row in which the optimum occurs is asymptotically forced back toward the first admissible row. The bridge between these global statements and the finite combinatorics is an exact prime-power exclusion principle.

For $1\le n\le|\mathcal P_r|$, write
\[
A_r(n)=\min_{\substack{S\subseteq\mathcal P_r\\|S|=n}}\operatorname{lcm}(S).
\]
For a prime $p$, put
\[
E_p(r)=\max_{0\le k\le\lfloor r/2\rfloor}v_p\binom rk,
\qquad
H_{p,e}(r)=\left\{0\le k\le\left\lfloor\frac r2\right\rfloor:v_p\binom rk>e\right\}.
\]

\begin{theorem}[Prime-power exclusion principle]\label{thm:exclusion}
Let $d_r=\lfloor r/2\rfloor+1$. Then
\begin{equation}
\boxed{
A_r(n)=
\min_{\substack{0\le e_p\le E_p(r)\\
|\bigcup_{p\le r}H_{p,e_p}(r)|\le d_r-n}}
\prod_{p\le r}p^{e_p}.}
\tag{1.4}
\end{equation}
\end{theorem}

Theorem \ref{thm:exclusion} is the structural core of the paper. It replaces a search over subsets of binomial coefficients by an exact optimization over prime-power budgets: lowering the permitted exponent of a prime power saves multiplicative cost but deletes a precisely determined set of indices. Section 4 derives the corresponding digital valuation rule and uses $n=15$ as a guiding example.

Our first global theorem shows that, despite this freedom, the leading growth is universal.

\begin{theorem}[Asymptotic growth]\label{thm:main}
Put
\[
\Phi(x)=\frac{(\log x)^{3/5}}{(\log\log x)^{1/5}}.
\]
There exists an absolute constant $c>0$ such that
\begin{equation}
\boxed{\log a(n)=2n+O\!\left(ne^{-c\Phi(n)}\right).}
\tag{1.6}
\end{equation}
Consequently
\[
\log a(n)\sim2n,
\qquad
a(n)^{1/n}\longrightarrow e^2.
\]
\end{theorem}

The quantitative error term uses the classical Korobov--Vinogradov form of the prime number theorem. For the qualitative conclusions $\log a(n)\sim2n$ and $a(n)^{1/n}\to e^2$, the ordinary prime number theorem is sufficient.

The upper bound comes immediately from the complete row $2n-2$ and \eqref{eq:farhi}. The lower bound is the main point. For primes $p>\sqrt r$, Lucas' theorem identifies intervals of primes that must divide every admissible LCM. In the $j$th large-prime band the forced interval has exact length
\[
\frac{2n-j-3}{j(j+1)}.
\]
These lengths telescope to $2n$; the Korobov--Vinogradov prime number theorem supplies the quantitative error term.

The same prime-band geometry controls not only the value of the optimum, but also where it occurs. In this sense the optimization is asymptotically rigid: an optimal row is forced back to the first admissible scale, and an optimal support can omit only a sublinear number of positions. Thus Theorem \ref{thm:main} determines the \emph{size} of the optimum, while the next theorem determines its \emph{location}. If $r_n$ is an optimal row and $\delta(n)$ is the minimum prefix defect of an optimal support, we obtain the following second main theorem.

\begin{theorem}[Localization of optimal rows and supports]\label{thm:rowlocal}
There exists an absolute constant $c>0$ such that every optimal row $r_n$ satisfies
\[
\boxed{r_n=2n+O\!\left(ne^{-c\Phi(n)}\right).}
\]
Moreover,
\[
\boxed{\delta(n)=O\!\left(ne^{-c\Phi(n)}\right)=o(n).}
\]
In particular, $r_n/n\to2$.
\end{theorem}

The final part of the paper returns to exact finite structure. For a support $S\subseteq\{0,\ldots,\lfloor r/2\rfloor\}$, define
\[
\delta(S)=1+\max S-|S|,
\]
the number of missing indices between $0$ and $\max S$; thus a prefix has defect $0$. Let $\delta(n)$ be the minimum defect among supports attaining $a(n)$.

Exact certification gives two illustrative finite transitions: $n=15$ is the first non-prefix optimum, and at $n=41$ the minimum defect is already $3$. These are the only computer-assisted statements in the paper; the analytic results are independent of them.

The paper proceeds from local $p$-adic structure to global optimization and then returns briefly to finite support geometry. Section 3 gives the finite reduction, Section 4 develops prime-power exclusion, Section 5 proves asymptotic growth and localization, and Section 6 records the exact finite examples.

\section{Preliminaries}
\subsection{\texorpdfstring{$p$}{p}-adic valuations}
For a prime $p$, let $v_p(m)$ be the exponent of $p$ in $m$. Legendre's formula gives
\begin{equation}
v_p(m!)=\sum_{\nu\ge1}\left\lfloor\frac m{p^\nu}\right\rfloor,
\tag{2.1}
\end{equation}
and hence
\begin{equation}
v_p\binom rk
=\sum_{\nu\ge1}
\left(
\left\lfloor\frac r{p^\nu}\right\rfloor-
\left\lfloor\frac k{p^\nu}\right\rfloor-
\left\lfloor\frac{r-k}{p^\nu}\right\rfloor
\right).
\tag{2.2}\label{eq:vpbinom}
\end{equation}
We use Kummer's and Lucas' theorems in their standard forms.

\subsection{Distinct entries in a row}
By symmetry, $\binom rk=\binom r{r-k}$, and
\[
\frac{\binom r{k+1}}{\binom rk}=\frac{r-k}{k+1}>1
\]
for $0\le k<\lfloor r/2\rfloor$. Thus $\Pcal_r$ in \eqref{eq:Pr} is exactly the set of distinct entries and
\begin{equation}
d_r:=|\Pcal_r|=\left\lfloor\frac r2\right\rfloor+1.
\tag{2.3}
\end{equation}
In particular, a row can supply $n$ distinct entries only if
\begin{equation}
r\ge2n-2.
\tag{2.4}\label{eq:rmin}
\end{equation}

\subsection{The full-row least common multiple}
Let
\[
\Lambda_r=\lcmop_{0\le k\le r}\binom rk.
\]
By \eqref{eq:farhi},
\begin{equation}
\Lambda_r=\frac{\lcmop(1,2,\ldots,r+1)}{r+1}.
\tag{2.5}\label{eq:fullrow}
\end{equation}
Taking $p$-adic valuations gives
\begin{proposition}\label{prop:Ep}
For every prime $p$,
\[
\boxed{E_p(r)=\lfloor\log_p(r+1)\rfloor-v_p(r+1).}
\]
\end{proposition}

\subsection{Chebyshev functions}
Let
\[
\theta(x)=\sum_{p\le x}\log p,
\qquad
\psi(x)=\sum_{p^\nu\le x}\log p.
\]
Then
\[
\psi(x)=\log\lcmop(1,2,\ldots,\lfloor x\rfloor).
\]
We use the classical Korobov--Vinogradov zero-free-region consequence
\begin{equation}
\theta(x)=x+O\!\left(xe^{-c_0\Phi(x)}\right),
\qquad
\psi(x)=x+O\!\left(xe^{-c_0\Phi(x)}\right)
\tag{2.6}\label{eq:pnt}
\end{equation}
for some absolute $c_0>0$; see, for example, \cite[Ch.~6]{MontgomeryVaughan2006}.

\section{Basic bounds and finite reduction}
Recall
\[
A_r(n)=\min_{\substack{S\subseteq\mathcal P_r\\|S|=n}}\operatorname{lcm}(S),
\qquad
a(n)=\min_{r\ge2n-2}A_r(n).
\]
\begin{lemma}\label{lem:largest}
Every $n$-element subset of $\mathcal P_r$ contains an element at least $\binom r{n-1}$. Hence
\begin{equation}
A_r(n)\ge\binom r{n-1}.
\tag{3.1}
\end{equation}
\end{lemma}
\begin{proof}
The half-row entries are strictly increasing, so the $n$th smallest is $\binom r{n-1}$. The LCM is at least the largest selected element.
\end{proof}

Taking $r=2n-2$ and using \eqref{eq:fullrow} gives the two elementary bounds
\begin{equation}
\boxed{
\binom{2n-2}{n-1}
\le a(n)\le
\frac{\operatorname{lcm}(1,2,\ldots,2n-1)}{2n-1}.}
\tag{3.2}\label{eq:basicbounds}
\end{equation}
The same lower bound makes the global minimization finite.

\begin{theorem}[Finite row reduction]\label{thm:cutoff}
Suppose $U\ge a(n)$. If
\[
\binom r{n-1}>U,
\]
then row $r$ cannot attain $a(n)$. In particular, with
\[
U_n=\frac{\operatorname{lcm}(1,\ldots,2n-1)}{2n-1}
\]
and
\[
R(n)=\max\left\{r\ge2n-2:\binom r{n-1}\le U_n\right\},
\]
one has
\[
\boxed{a(n)=\min_{2n-2\le r\le R(n)}A_r(n).}
\]
\end{theorem}
\begin{proof}
By Lemma \ref{lem:largest}, $A_r(n)\ge\binom r{n-1}$. The special choice $U=U_n$ is admissible by \eqref{eq:basicbounds}.
\end{proof}

Theorem \ref{thm:cutoff} is used twice: asymptotically, to constrain the location of an optimal row, and computationally, to turn every finite claim into an exact finite certificate. The certification algorithm itself is postponed to Section \ref{sec:finite}, where it is needed.

\section{Prime-power exclusion}
We now prove Theorem \ref{thm:exclusion}.

\begin{lemma}\label{lem:digitalvp}
For every prime $p$ and $0\le k\le r$,
\[
\boxed{
v_p\binom rk
=\sum_{\nu\ge1}\1_{\{k\bmod p^\nu>r\bmod p^\nu\}}.}
\]
\end{lemma}
\begin{proof}
Fix $q=p^\nu$, and write $r=aq+s$ and $k=bq+t$ with $0\le s,t<q$. The corresponding summand of \eqref{eq:vpbinom} is
\[
a-b-\left\lfloor\frac{r-k}{q}\right\rfloor.
\]
If $t\le s$, then $r-k=(a-b)q+(s-t)$ and the summand is $0$. If $t>s$, then $r-k=(a-b-1)q+(q+s-t)$ and the summand is $1$. Thus the summand is exactly $\1_{\{t>s\}}$. Summing over $\nu$ gives the formula.
\end{proof}

\begin{proof}[Proof of Theorem \ref{thm:exclusion}]
For exponents $0\le e_p\le E_p(r)$, put $M=\prod_{p\le r}p^{e_p}$. Then
\[
\binom rk\mid M
\quad\Longleftrightarrow\quad
v_p\binom rk\le e_p\ \text{for every }p\le r
\quad\Longleftrightarrow\quad
k\notin\bigcup_{p\le r}H_{p,e_p}(r).
\]
Hence the number of distinct half-row coefficients dividing $M$ is exactly
\[
d_r-\left|\bigcup_{p\le r}H_{p,e_p}(r)\right|.
\]
Every admissible exponent vector therefore gives an integer divisible by at least $n$ distinct entries.

Conversely, if $S$ is an $n$-element support and $L=\lcmop\{\binom rk:k\in S\}$, then $e_p=v_p(L)$ gives an admissible exponent vector, none of the indices in $S$ lies in the union of the corresponding exclusion sets, and $L=\prod p^{e_p}$. The two minima coincide.
\end{proof}

\begin{remark}
Theorem \ref{thm:exclusion} separates multiplicative cost from combinatorial loss. Lowering the allowed exponent of $p$ removes precisely the indices in $H_{p,e}(r)$; several such choices combine by taking a union. This is the mechanism behind the holes in optimal supports.
\end{remark}

\begin{example}[The first hole]\label{ex:n15}
Row $34$ has $18$ distinct entries. Its full-row LCM is $\Lambda_{34}$, but lowering the $3$-adic budget from $3$ to $2$ excludes
\[
H_{3,2}(34)=\{8,17\},
\]
while removing the factor $19$ excludes
\[
H_{19,0}(34)=\{16,17\}.
\]
The union has only three indices,
\[
\{8,16,17\},
\]
so the remaining $15$ coefficients all divide
\[
\frac{\Lambda_{34}}{3\cdot19}=72\,382\,733\,280.
\]
Equivalently, the surviving support is
\[
\{0,1,\ldots,7,9,10,\ldots,15\}.
\]
Section \ref{sec:finite} will certify that this candidate is globally optimal and that $n=15$ is the first value for which no prefix is optimal. The present example isolates the arithmetic mechanism: one strategically omitted coefficient allows an expensive prime-power requirement to disappear.
\end{example}

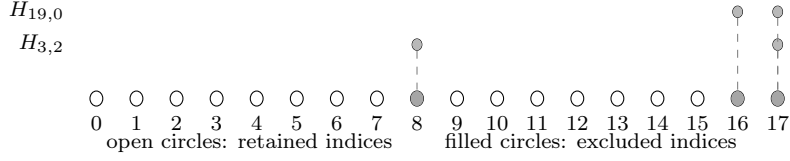
\begin{figure}[htbp]
\centering
\begin{tikzpicture}[x=0.53cm,y=0.62cm,font=\scriptsize]
  \foreach \k in {0,...,17}{
    \draw (\k,0) circle (0.16);
    \node[below=3pt] at (\k,0) {$\k$};
  }
  \foreach \k in {8,16,17}{\fill[gray!70] (\k,0) circle (0.16);}

  \node[anchor=east] at (-0.55,1.15) {$H_{3,2}$};
  \foreach \k in {8,17}{
    \draw (\k,1.15) circle (0.12);
    \fill[gray!55] (\k,1.15) circle (0.12);
    \draw[dashed,gray] (\k,1.02) -- (\k,0.19);
  }
  \node[anchor=east] at (-0.55,1.85) {$H_{19,0}$};
  \foreach \k in {16,17}{
    \draw (\k,1.85) circle (0.12);
    \fill[gray!55] (\k,1.85) circle (0.12);
    \draw[dashed,gray] (\k,1.72) -- (\k,0.19);
  }
  \node[anchor=west,align=left] at (0,-0.95)
    {open circles: retained indices \qquad filled circles: excluded indices};
\end{tikzpicture}
\caption{Prime-power exclusion in the first non-prefix optimum.  In row $34$, $H_{3,2}=\{8,17\}$ and $H_{19,0}=\{16,17\}$ overlap at index $17$.  Thus two multiplicative savings cost only three indices in total.  This overlap is what makes the non-prefix support competitive.}
\label{fig:n15exclusion}
\end{figure}

Figure \ref{fig:n15exclusion} is a small instance of the general principle behind Theorem \ref{thm:exclusion}: multiplicative savings are governed not only by the sizes of exclusion sets, but also by how strongly those sets overlap.

\section{Asymptotic growth}
We now prove Theorem \ref{thm:main}. The argument has two layers. Large primes divide Pascal coefficients according to particularly simple intervals of indices. Some corresponding exclusion sets are too large for an $n$-element support to avoid, so those primes are forced into every admissible LCM; their logarithms supply the main term $2n$. For localization we then use two adjacent prime bands in which individual primes may be omitted, but not too many simultaneously because their exclusion intervals overlap rigidly. The first layer determines the size of $a(n)$; the second determines where an optimum can occur.

\subsection{A coarse row bound}
\begin{lemma}\label{lem:localize}
For all sufficiently large $n$, every row $r$ satisfying $A_r(n)=a(n)$ satisfies
\[
2n-2\le r<4n.
\]
\end{lemma}
\begin{proof}
The lower bound is \eqref{eq:rmin}. For the upper bound, \eqref{eq:basicbounds} and \eqref{eq:pnt} give
\[
\log a(n)\le\psi(2n-1)-\log(2n-1)=2n+o(n).
\]
On the other hand, if $r\ge4n$, then by Lemma \ref{lem:largest},
\[
A_r(n)\ge\binom{4n}{n-1}.
\]
Stirling's formula yields
\[
\log\binom{4n}{n-1}=(4\log4-3\log3)n+O(\log n),
\]
and $4\log4-3\log3=2.24934\ldots>2$, a contradiction for large $n$.
\end{proof}

\subsection{Forced primes}
\begin{lemma}[Forced-prime count]\label{lem:forced}
Let $p>\sqrt r$ be an odd prime and write
\[
r=jp+s,\qquad0\le s<p.
\]
Then the number of distinct entries of row $r$ that are not divisible by $p$ is
\begin{equation}
\boxed{\left\lceil\frac{(j+1)(s+1)}2\right\rceil.}
\tag{5.1}\label{eq:halfcount}
\end{equation}
Consequently, $p$ divides the LCM of every $n$-element subset of $\Pcal_r$ if and only if
\begin{equation}
\boxed{(j+1)(s+1)\le2n-2.}
\tag{5.2}\label{eq:forcedcondition}
\end{equation}
\end{lemma}
\begin{proof}
Since $p>\sqrt r$, one has $r<p^2$, so $r=jp+s$ is the base-$p$ expansion of $r$. Writing $k=up+v$, Lucas' theorem gives
\[
p\nmid\binom rk
\quad\Longleftrightarrow\quad
u\le j\ \text{and}\ v\le s.
\]
Thus the full row has
\[
T=(j+1)(s+1)
\]
$p$-free indices. They are paired by $k\leftrightarrow r-k$, except possibly for the central index when $r$ is even. If $j,s$ are both even, then $T$ is odd and $r/2=(j/2)p+s/2$ is $p$-free; if $j,s$ are both odd, then $T$ is even and the lower base-$p$ digit $(p+s)/2$ of $r/2$ exceeds $s$, so the central coefficient is divisible by $p$. Hence in every case the half-row contains exactly $\lceil T/2\rceil$ $p$-free entries, proving \eqref{eq:halfcount}.

The prime $p$ is forced precisely when fewer than $n$ distinct entries are $p$-free. Since $T$ is integral,
\[
\left\lceil\frac T2\right\rceil<n
\quad\Longleftrightarrow\quad
T\le2n-2,
\]
which is \eqref{eq:forcedcondition}.
\end{proof}

\begin{corollary}[Forced intervals]\label{cor:intervals}
Let $r\ge2n-2$ and $j\ge1$. Put
\[
U_j=\frac rj,
\qquad
L_j=\frac{r+1}{j}-\frac{2n-2}{j(j+1)}.
\]
If $p>\sqrt r$ is prime and $L_j\le p\le U_j$, then $p$ divides every admissible $n$-element LCM in row $r$. Moreover
\begin{equation}
L_j>\frac r{j+1}
\tag{5.3}\label{eq:bandleft}
\end{equation}
and
\begin{equation}
\boxed{U_j-L_j=\frac{2n-j-3}{j(j+1)}.}
\tag{5.4}\label{eq:bandlength}
\end{equation}
\end{corollary}
\begin{proof}
Within $r/(j+1)<p\le r/j$ one has $j=\lfloor r/p\rfloor$ and $s=r-jp$. Condition \eqref{eq:forcedcondition} becomes
\[
(j+1)(r-jp+1)\le2n-2,
\]
which is equivalent to $p\ge L_j$. Also
\[
L_j-\frac r{j+1}=\frac{r+j+3-2n}{j(j+1)}\ge\frac{j+1}{j(j+1)}>0
\]
by $r\ge2n-2$. Finally, direct subtraction gives \eqref{eq:bandlength}.
\end{proof}

The constant $2$ in Theorem \ref{thm:main} is not an artifact of the complete-row upper bound; it is already encoded in the forced-prime geometry. Ignoring lower-order terms, the $j$th band contributes a forced interval of length
\[
\frac{2n}{j(j+1)},
\]
and
\[
\sum_{j\ge1}\frac{1}{j(j+1)}=1.
\]
In other words, the prime bands partition an asymptotic logarithmic mass of $2n$: each band contributes its share, and the shares telescope to the full main term. The proof below makes this principle uniform while allowing the number of bands to grow with $n$.

\begin{figure}[htbp]
\centering
\begin{tikzpicture}[x=0.88cm,y=0.9cm,font=\small,
  endpoint/.style={font=\scriptsize,inner sep=1pt},
  bandlabel/.style={font=\small,anchor=east},
  interval/.style={line width=6pt,line cap=butt}]
  \def\yA{1.65}
  \node[bandlabel] at (-0.55,\yA) {first band};
  \draw[thin] (0,\yA) -- (10.6,\yA);
  \draw[interval,gray!35] (1.55,\yA) -- (4.45,\yA);
  \draw[interval,gray!72] (6.05,\yA) -- (10.6,\yA);
  \node[font=\scriptsize] at (3.00,\yA) {may be omitted};
  \node[font=\scriptsize,text=white] at (8.32,\yA) {forced};
  \foreach \x in {0,1.55,4.45,6.05,10.6} {\draw (\x,\yA-0.11)--(\x,\yA+0.11);}
  \node[endpoint,below=3pt] at (0,\yA) {$r/2$};
  \node[endpoint,below=3pt] at (1.55,\yA) {$b$};
  \node[endpoint,below=3pt] at (4.45,\yA) {$b+q$};
  \node[endpoint,below=3pt] at (6.05,\yA) {$L_1$};
  \node[endpoint,below=3pt] at (10.6,\yA) {$r$};

  \def\yB{-0.15}
  \node[bandlabel] at (-0.55,\yB) {second band};
  \draw[thin] (0,\yB) -- (10.6,\yB);
  \draw[interval,gray!35] (1.55,\yB) -- (4.45,\yB);
  \draw[interval,gray!72] (6.05,\yB) -- (10.6,\yB);
  \node[font=\scriptsize] at (3.00,\yB) {may be omitted};
  \node[font=\scriptsize,text=white] at (8.32,\yB) {forced};
  \foreach \x in {0,1.55,4.45,6.05,10.6} {\draw (\x,\yB-0.11)--(\x,\yB+0.11);}
  \node[endpoint,below=3pt] at (0,\yB) {$r/3$};
  \node[endpoint,below=3pt] at (1.55,\yB) {$(r+1)/3$};
  \node[endpoint,below=3pt] at (4.45,\yB) {$(r+q+1)/3$};
  \node[endpoint,below=3pt] at (6.05,\yB) {$L_2$};
  \node[endpoint,below=3pt] at (10.6,\yB) {$r/2$};
\end{tikzpicture}
\caption{The first two large-prime bands (schematic, not to scale). Dark intervals are individually forced by Corollary \ref{cor:intervals}. Light intervals contain primes that a support may omit only when the corresponding exclusion sets fit inside its $q$ omitted positions. The two light masses cannot be removed simultaneously; this is the source of the localization gain.}
\label{fig:twobands}
\end{figure}
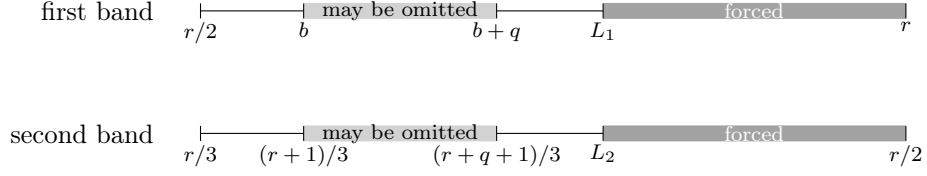

The dark intervals in Figure \ref{fig:twobands} drive Theorem \ref{thm:main}; the incompatibility of the two light intervals drives Theorem \ref{thm:rowlocal}.

\subsection{Proof of the asymptotic theorem}
\begin{proof}[Proof of Theorem \ref{thm:main}]
For the upper bound, row $2n-2$ and \eqref{eq:fullrow} give
\[
a(n)\le\frac{\lcmop(1,\ldots,2n-1)}{2n-1},
\]
so by \eqref{eq:pnt},
\begin{equation}
\log a(n)\le2n+O\!\left(ne^{-c_1\Phi(n)}\right).
\tag{5.5}\label{eq:uppermain}
\end{equation}

For the lower bound, let $r$ be an optimal row. By Lemma \ref{lem:localize}, $2n-2\le r<4n$ for sufficiently large $n$. Choose a sufficiently small fixed $\eta>0$ and set
\[
J=\left\lfloor e^{\eta\Phi(n)}\right\rfloor.
\]
Then $J=n^{o(1)}$, hence $J<\sqrt r-1$ for large $n$. For $1\le j\le J$, let $I_j=(L_j,U_j]$. By \eqref{eq:bandleft},
\[
L_j>\frac r{j+1}\ge\frac r{J+1}>\sqrt r
\]
for large $n$. Every prime in $I_j$ is therefore covered by Lemma \ref{lem:forced} and divides $a(n)$. The intervals lie in disjoint bands, so
\begin{equation}
\log a(n)\ge\sum_{j=1}^J\bigl(\theta(U_j)-\theta(L_j)\bigr).
\tag{5.6}\label{eq:thetasum}
\end{equation}

By \eqref{eq:bandlength},
\[
\sum_{j=1}^J(U_j-L_j)
=\sum_{j=1}^J\frac{2n-j-3}{j(j+1)}
=(2n-2)\frac J{J+1}-H_J,
\]
where $H_J$ is the $J$th harmonic number. Thus
\begin{equation}
\sum_{j=1}^J(U_j-L_j)=2n+O\!\left(\frac nJ+\log J\right).
\tag{5.7}
\end{equation}

Every endpoint in \eqref{eq:thetasum} lies between $(2n-2)/(J+1)$ and $4n$. Since $\log J=O(\Phi(n))=o(\log n)$, one has $\Phi(x)=(1+o(1))\Phi(n)$ uniformly throughout this range. Hence, after decreasing the positive constant if necessary, \eqref{eq:pnt} gives
\[
\theta(x)=x+O\!\left(xe^{-c_2\Phi(n)}\right)
\]
uniformly at all endpoints. As $L_j<U_j=r/j$ and $r<4n$,
\[
\sum_{j\le J}(U_j+L_j)=O(n\log J),
\]
so the total PNT error is $O(ne^{-c_3\Phi(n)})$. Also $n/J=O(ne^{-\eta\Phi(n)})$, while $\log J$ is negligible on this scale. Therefore
\begin{equation}
\log a(n)\ge2n-O\!\left(ne^{-c_4\Phi(n)}\right).
\tag{5.8}\label{eq:lowermain}
\end{equation}
Combining \eqref{eq:uppermain} and \eqref{eq:lowermain} proves the theorem.
\end{proof}

\subsection{A two-band refinement and localization of optimal rows}
We now use information that was deliberately discarded in the proof of Theorem \ref{thm:main}: primes that are not individually forced can still fail to be omitted simultaneously, because their exclusion sets must all fit inside the same set of omitted indices.

Let $r$ satisfy $2n-2\le r<4n$, put
\[
m=\left\lfloor\frac r2\right\rfloor,
\qquad
q=m+1-n,
\qquad
b=\left\lceil\frac r2\right\rceil,
\]
and let $S\subseteq\{0,1,\ldots,m\}$ have $|S|=n$. Write
\[
Q=\{0,1,\ldots,m\}\setminus S,
\qquad |Q|=q.
\]
For a prime $p$, let
\[
H_p=\left\{0\le k\le m:p\mid\binom rk\right\}.
\]
Then $p$ is absent from the LCM associated with $S$ if and only if $H_p\subseteq Q$.

For any nested family $H_s\subseteq H_t$ when $s\le t$, the members contained in a fixed set $Q$ form an initial segment; hence their union is the largest one. We use this elementary observation below.

The first two prime bands contain particularly simple nested families of exclusion sets. The next lemma records their exact geometry. Since it is used only in the asymptotic argument, we assume $n\ge4$; then every prime in the second band below is odd.

\begin{lemma}[Geometry of the first two bands]\label{lem:twobands}
Assume $n\ge4$ and $2n-2\le r<4n$. Let $S,Q,m,q,b$ be as above.

In the first band, consider primes
\[
b<p\le b+q.
\]
For such a prime,
\[
H_p=\{r-p+1,\ldots,m\},
\qquad |H_p|=p-b.
\]
Let $h_1$ be the largest value of $|H_p|$ among primes in this interval that are absent from the LCM of $S$, and put $h_1=0$ if there is no such prime.

In the second band, consider primes
\[
\frac{r+1}{3}<p\le\frac{r+q+1}{3}.
\]
For such a prime,
\[
H_p=\{r-2p+1,\ldots,p-1\},
\qquad |H_p|=3p-r-1.
\]
Let $h_2$ be the largest value of $|H_p|$ among primes in this interval that are absent from the LCM of $S$, and put $h_2=0$ if there is no such prime.

Then
\begin{equation}
\boxed{
 h_1+\frac{h_2}{3}
 \le
 q+\frac{(2q-n)_+}{6},
}
\tag{5.9}\label{eq:twobandgeometry}
\end{equation}
where $(x)_+=\max(x,0)$.
\end{lemma}

\begin{proof}
For $b<p\le b+q$ we have $r=p+s$ with $m<p$. Lucas' theorem gives $p\nmid\binom rk$ exactly for $k\le s$, and hence
\[
H_p=\{r-p+1,\ldots,m\},\qquad |H_p|=p-b.
\]
These sets increase with $p$. Thus, if $h_1>0$, the union of the exclusion sets of all omitted primes in the first band is
\[
A=\{m-h_1+1,\ldots,m\}\subseteq Q;
\]
for $h_1=0$ set $A=\varnothing$.

For $(r+1)/3<p\le(r+q+1)/3$ we have $r=2p+s$ with $0\le s<p$ and $m=p+\lfloor s/2\rfloor<2p$. Lucas' theorem shows that divisibility by $p$ occurs in the half-row exactly for $s+1\le k\le p-1$. Hence
\[
H_p=\{r-2p+1,\ldots,p-1\},\qquad |H_p|=3p-r-1.
\]
These sets also increase with $p$. Therefore, if $h_2>0$, their union over omitted primes is
\[
B=\left\{\frac{r-2h_2+1}{3},\ldots,\frac{r+h_2-2}{3}\right\}\subseteq Q;
\]
for $h_2=0$ set $B=\varnothing$.

It remains to exploit the fact that both $A$ and $B$ must fit inside the same $q$-element set $Q$.

If $A$ and $B$ are disjoint, then
\[
h_1+h_2=|A\cup B|\le q,
\]
so \eqref{eq:twobandgeometry} follows immediately.

Assume now that $A$ and $B$ overlap. Write
\[
r=2m+\varepsilon,
\qquad \varepsilon\in\{0,1\},
\]
so that
\[
m=n+q-1.
\]
Let
\[
a=m-h_1+1
\]
be the left endpoint of $A$ and
\[
\ell=\frac{r-2h_2+1}{3}
\]
the left endpoint of $B$.

First suppose $a\le\ell$. Since $A$ ends at $m$ and $B$ ends at most at $m$, the overlap together with $a\le\ell$ gives $B\subseteq A$. The inequality $a\le\ell$ is equivalent to
\[
2h_2\le3h_1-n-q-1+\varepsilon.
\]
Using $h_1\le q$ and $\varepsilon\le1$, we obtain
\[
 h_1+\frac{h_2}{3}
 \le
 q+\frac{2q-n}{6}
 \le
 q+\frac{(2q-n)_+}{6}.
\]

It remains to consider $\ell<a$. Since the intervals overlap, $A\cup B$ is the interval from $\ell$ to $m$. The inclusion $A\cup B\subseteq Q$ therefore gives
\[
m-\ell+1\le q,
\]
or equivalently $\ell\ge n$. Substituting the value of $\ell$ yields
\[
2h_2\le2q-n-1+\varepsilon\le2q-n.
\]
In particular this case is possible only when $2q\ge n$. Since $h_1\le q$,
\[
 h_1+\frac{h_2}{3}
 \le
 q+\frac{2q-n}{6}
 =
 q+\frac{(2q-n)_+}{6}.
\]
This proves the lemma.
\end{proof}

\begin{proposition}[Refined row-wise lower bound]\label{prop:refinedlower}
There exists an absolute constant $c>0$ such that, uniformly for $2n-2\le r<4n$ and
\[
q=\left\lfloor\frac r2\right\rfloor+1-n,
\]
one has
\begin{equation}
\boxed{
\log A_r(n)
\ge
2n+\min\!\left(\frac q3,\frac n6\right)
-O\!\left(ne^{-c\Phi(n)}\right).
}
\tag{5.10}\label{eq:refinedlower}
\end{equation}
\end{proposition}

\begin{proof}
Fix an $n$-element support $S\subseteq\{0,\ldots,m\}$ and let $L(S)$ denote the LCM of the corresponding coefficients. We prove the asserted lower bound for $\log L(S)$ uniformly in $S$.

The forced-prime argument used in the proof of Theorem \ref{thm:main} is row-wise: uniformly for every $2n-2\le r<4n$, the primes in the intervals $I_j$ with $1\le j\le J$ contribute
\begin{equation}
2n-O\!\left(ne^{-c_1\Phi(n)}\right)
\tag{5.11}\label{eq:forceduniform}
\end{equation}
to $\log L(S)$. We now add primes from the parts of the first two bands immediately below the corresponding forced intervals.

Put
\[
b=\left\lceil\frac r2\right\rceil.
\]
For the first band put
\[
\mathcal E_1=\{p\text{ prime}: b<p\le b+q\}.
\]
By Lemma \ref{lem:twobands}, the omitted primes in $\mathcal E_1$ form the initial segment $b<p\le b+h_1$. Hence the retained logarithmic mass is
\[
\theta(b+q)-\theta(b+h_1)
=q-h_1+O\!\left(ne^{-c_2\Phi(n)}\right).
\]
Likewise, for
\[
\mathcal E_2=\left\{p\text{ prime}:\frac{r+1}{3}<p\le\frac{r+q+1}{3}\right\},
\]
the omitted primes are those up to $(r+h_2+1)/3$, and the retained mass is
\[
\theta\!\left(\frac{r+q+1}{3}\right)-
\theta\!\left(\frac{r+h_2+1}{3}\right)
=\frac{q-h_2}{3}+O\!\left(ne^{-c_2\Phi(n)}\right).
\]
All endpoints are comparable with $n$. These prime sets are disjoint from each other and from the forced intervals in \eqref{eq:forceduniform}: in the first band $b+q=r+1-n<L_1$, while in the second, writing $r=2m+\varepsilon$, one has
\[
L_2-\frac{r+q+1}{3}=\frac{\varepsilon+1}{6}>0.
\]
Thus the extra retained prime mass is
\[
\frac{4q}{3}
-
\left(h_1+\frac{h_2}{3}\right)
-O\!\left(ne^{-c_2\Phi(n)}\right).
\]
Lemma \ref{lem:twobands} therefore gives the lower bound
\[
\frac{4q}{3}
-
q-\frac{(2q-n)_+}{6}
-O\!\left(ne^{-c_2\Phi(n)}\right).
\]
The main term equals
\[
\begin{cases}
q/3,&2q\le n,\\[2mm]
n/6,&2q\ge n,
\end{cases}
\]
that is,
\[
\min\!\left(\frac q3,\frac n6\right).
\]
Combining this with \eqref{eq:forceduniform} proves \eqref{eq:refinedlower} for every support $S$, and hence for $A_r(n)$.
\end{proof}

\begin{proof}[Proof of Theorem \ref{thm:rowlocal}]
Let $r_n$ be any optimal row and put
\[
q_n=\left\lfloor\frac{r_n}{2}\right\rfloor+1-n.
\]
By Lemma \ref{lem:localize}, $2n-2\le r_n<4n$ for large $n$, so Proposition \ref{prop:refinedlower} applies. Combining it with the upper bound \eqref{eq:uppermain} gives
\[
\min\!\left(\frac{q_n}{3},\frac n6\right)
=O\!\left(ne^{-c\Phi(n)}\right)
\]
for some absolute $c>0$. Since the right side is $o(n)$, the second alternative is impossible for sufficiently large $n$. Hence
\[
q_n=O\!\left(ne^{-c\Phi(n)}\right).
\]
Because
\[
q_n=\left\lfloor\frac{r_n}{2}\right\rfloor+1-n,
\]
we obtain
\[
r_n=2n+O\!\left(ne^{-c\Phi(n)}\right).
\]

Finally, every $n$-element support in row $r_n$ has prefix defect at most
\[
d_{r_n}-n=q_n.
\]
Taking the minimum over optimal supports gives
\[
\delta(n)=O\!\left(ne^{-c\Phi(n)}\right)=o(n).
\]
\end{proof}

\section{Computer-certified finite structure}\label{sec:finite}
The analytic results above are independent of computation. We now return to the finite support geometry suggested by Example \ref{ex:n15} and certify the first structural transitions exactly.

For $S\subseteq\{0,\ldots,\lfloor r/2\rfloor\}$, recall
\[
\delta(S)=1+\max S-|S|.
\]
For an $n$-element support, the condition $\delta(S)\le t$ is equivalent to $\max S\le n+t-1$. Thus defect-restricted minima are obtained simply by restricting the allowed indices.

\subsection{Exact certification within a row}
For fixed $r$, write $c_k=\binom rk$, $0\le k<d_r$. For a finite set $X$ of positive integers, let
\[
\operatorname{MinDiv}(X)=\{x\in X:\text{no distinct }y\in X\text{ satisfies }y\mid x\}.
\]
Set
\[
D_{0,0}=\{1\},\qquad D_{0,t}=\varnothing\quad(t>0),
\]
and recursively
\begin{equation}
D_{j+1,t}=\operatorname{MinDiv}\!\left(D_{j,t}\cup\{\operatorname{lcm}(L,c_j):L\in D_{j,t-1}\}\right),
\tag{6.1}\label{eq:dprec}
\end{equation}
with the second set omitted for $t=0$.

\begin{theorem}[Exact LCM-state dynamic program]\label{thm:lcmdp}
For $0\le t\le d_r$,
\[
\boxed{A_r(t)=\min D_{d_r,t},}
\]
with $A_r(0)=1$.
\end{theorem}
\begin{proof}
Without pruning, \eqref{eq:dprec} is the standard subset recurrence. If $L_1\mid L_2$, then for every positive integer $x$,
\[
\operatorname{lcm}(L_1,x)\mid\operatorname{lcm}(L_2,x).
\]
Hence a state divisible by another state of the same cardinality can never yield a smaller future LCM and may be discarded. Induction on $j$ proves the result.
\end{proof}

Together with the row cutoff from Theorem \ref{thm:cutoff}, Theorem \ref{thm:lcmdp} yields a finite exact certificate for every value used below. An independent second implementation is included in the supplementary archive.

\subsection{The first non-prefix optimum}
The first global departure from prefix optimality occurs at $n=15$. Example \ref{ex:n15} already exhibited the winning support in row $34$; the purpose of the certification below is to prove that no earlier $n$ and no competing row does better.

\begin{proposition}[Computer-certified]\label{prop:n15}
For $1\le n\le14$, $\delta(n)=0$. At $n=15$,
\[
\delta(15)=1,
\qquad
a(15)=72\,382\,733\,280,
\]
with the optimal support from Example \ref{ex:n15}. The best prefix value is $77\,636\,318\,760$, attained in row $29$.
\end{proposition}
\begin{proof}[Certification]
The candidate from Example \ref{ex:n15} and
\[
\binom{43}{14}=78\,378\,960\,360>a(15)
\]
reduce the unrestricted search to $28\le r\le42$ by Theorem \ref{thm:cutoff}. The exact minima in these rows are listed in \texttt{row\_minima.csv} and are reproduced by both certification programs. The defect-restricted certificates show $\delta(n)=0$ for $n\le14$ and give the stated best prefix value at $n=15$.
\end{proof}

\subsection{The first genuinely multi-hole optimum}
The second structural transition occurs at $n=41$: the minimum defect jumps from the previously observed values $0$ and $1$ to $3$. Thus one missing index no longer suffices, and in fact neither do two.

\begin{proposition}[Computer-certified]\label{prop:n41}
Among $n\le40$, the values with $\delta(n)=1$ are exactly
\[
15,\ 23,\ 24,\ 36,
\]
and all remaining values have defect $0$. At $n=41$,
\[
\boxed{\delta(41)=3.}
\]
One optimum is attained in row $90$ by
\[
S_{41}=\{0,1,\ldots,43\}\setminus\{22,29,30\},
\]
with
\[
a(41)=\frac{\Lambda_{90}}{23\cdot31\cdot47}.
\]
Exactly five half-row coefficients fail to divide this integer, at indices
\[
22,29,30,44,45.
\]
Moreover,
\[
\frac{\Lambda_{90}}{a(41)}=23\cdot31\cdot47,
\]
and the corresponding exclusion sets are
\[
H_{23,0}(90)=\{22,45\},\qquad
H_{31,0}(90)=\{29,30\},\qquad
H_{47,0}(90)=\{44,45\}.
\]
Among all $41$-element supports of defect at most $2$, the least possible LCM is $13a(41)/4$, already attained with defect $1$ in row $91$.
\end{proposition}

\begin{proof}[Certification]
The displayed candidate satisfies
\[
\binom{122}{40}>a(41),
\]
so Theorem \ref{thm:cutoff} reduces the unrestricted check to $80\le r\le121$. Exact unrestricted and defect-restricted row minima are contained in the supplementary tables and are reproduced independently by both implementations. These data prove the claims about $\delta(n)$ through $41$ and show that defect $2$ does not improve the best defect-$1$ value at $n=41$.

The three displayed exclusion sets give the arithmetic explanation for the optimal pattern in row $90$: their union is exactly
\[
\{22,29,30,44,45\}.
\]
Thus the computation establishes global optimality, while Theorem \ref{thm:exclusion} explains the shape of the support.
\end{proof}

The two examples isolate the first two qualitative changes in optimal support geometry:
\[
\text{prefix}\quad\longrightarrow\quad\text{one hole}\quad\longrightarrow\quad\text{genuinely multi-hole}.
\]
They are not used in the asymptotic arguments of Section 5.

\paragraph{Certification and reproducibility.}
The finite assertions above were checked by two independent exact implementations: an LCM-state dynamic program with divisibility pruning and a prime-power exclusion verifier using excluded-index bitmasks. Complete row-by-row certificates, source code, outputs, and checksums are provided in the supplementary archive. No floating-point or heuristic comparisons are used.

\section{Discussion}
The results give three complementary descriptions of the same optimization problem. Theorem \ref{thm:exclusion} is local and arithmetic: it identifies exactly which coefficients are lost when a prime-power budget is lowered. Theorem \ref{thm:main} is global and analytic: after all possible local savings are allowed, the surviving large primes still contribute logarithmic mass $2n+o(n)$. Theorem \ref{thm:rowlocal} adds rigidity: an optimum cannot realize that saving far out in Pascal's triangle, but must occur at row $2n+o(n)$.

The finite examples then become structural rather than exceptional. The hole at $n=15$ and the multi-hole pattern at $n=41$ are concrete manifestations of the same prime-power tradeoff described by Theorem \ref{thm:exclusion}. Thus the paper moves from local $p$-adic constraints to global asymptotics and returns, with that structure in hand, to explain the geometry of optimal finite supports.

From the viewpoint of the classical full-row identity, allowing arbitrary subsets changes the finite structure but not the first-order asymptotic scale. The freedom to omit coefficients is genuine---already at $n=15$ it improves on every prefix---but asymptotically it cannot move an optimum far from the first admissible row or remove a positive proportion of the initial support. The prime-power exclusion principle gives the structural link between the full-row and subset problems.

\section{Outlook}
Theorems \ref{thm:main} and \ref{thm:rowlocal} determine the first-order scale of both the value and the location of the optimum. Three finer questions remain natural.

First, what is the true size of
\[
r_n-(2n-2)?
\]
Theorem \ref{thm:rowlocal} makes it $o(n)$, but the correct order may be far smaller. Second, how large can the prefix defect $\delta(n)$ become, and how often is it positive? The certified transitions at $15$ and $41$ show that non-prefix structure is genuine, while the localization theorem proves only that it is asymptotically sparse.

Both questions ask for information that is invisible in the leading term $2n$ and require a finer understanding of how the prime-power exclusion sets overlap near the first admissible rows.

\medskip
\textbf{Data and code availability.} All source code, exact output logs, machine-readable certificates, and SHA-256 checksums used to certify the finite results of Section~\ref{sec:finite} are provided in the supplementary reproducibility archive accompanying this submission. \\[-1mm]
\textbf{Declaration of generative AI and AI-assisted technologies.} During the development and preparation of this work, the author used OpenAI ChatGPT to assist with exploratory mathematical discussion, code drafting, organization of the manuscript, and language editing. The author reviewed and edited the resulting material, verified the mathematical content and references, and takes full responsibility for the content of the manuscript.

\end{document}